\documentclass{amsproc}
\usepackage{amssymb}
\usepackage{graphicx}
\usepackage{etex}
\usepackage{hyperref}                                   
\hypersetup{colorlinks=true}   
\usepackage[utf8x]{inputenc}
\usepackage{lscape}
\usepackage[toc,page]{appendix}
\usepackage{wrapfig}
\usepackage{indentfirst}
\usepackage{setspace}
\usepackage{amsmath,amsfonts,amssymb,amscd,amsthm,xspace}
\usepackage{listings}
\usepackage{comment}
\usepackage{amsmath}
\usepackage{longtable}

\numberwithin{equation}{section}
\usepackage{amssymb}
\usepackage{amsthm}
\usepackage[all,cmtip]{xy}
\usepackage[toc,page]{appendix}

\usepackage{mathrsfs} 
\usepackage{mathtools} 

\usepackage{bbm}
\usepackage{environ}

\usepackage{float}
\usepackage{amsmath}
\usepackage{longtable}
\usepackage{booktabs}
\numberwithin{equation}{section}
\usepackage{amssymb}
\usepackage{amsthm}
\usepackage[all,cmtip]{xy}
\usepackage[toc,page]{appendix}
\usepackage[british]{babel}
\usepackage{courier}
\usepackage{enumerate}
\usepackage{environ}
\usepackage[disable]{todonotes}
\usepackage{graphicx}
\usepackage{epstopdf}
\usepackage[scriptsize]{subfigure}
\usepackage{booktabs}
\usepackage{rotating}
\usepackage{listings}
\usepackage{tikz-cd}
\usepackage{lscape}
\usepackage{tkz-euclide}
\usepackage[toc,page]{appendix}
\usepackage{wrapfig}
\usepackage{indentfirst}

\usepackage[T1]{fontenc}
\usepackage{listings}
\usepackage{comment}

\usepackage{mathrsfs} 
\usepackage{mathtools} 
\usepackage{stmaryrd} 

\usepackage{bbm}

\newtheorem{theorem}{Theorem}[section]
\newtheorem{lemma}[theorem]{Lemma}

\theoremstyle{definition}
\newtheorem{definition}[theorem]{Definition}
\newtheorem{example}[theorem]{Example}

\theoremstyle{remark}

\numberwithin{equation}{section}

\begin{document} 
 
 \begin{abstract} 
 The modular group $\Gamma$ (which is the Hecke group $H_3$) can be used to study triangular maps,  Here we use the Hecke group $H_4$ to study the regular map that underlies Bring's surface of genus 4. Our main result is the determination of the 20-sided hyperbnolic polygon that gives Bring's surface.
 \end{abstract}

\title{Hecke--Farey coordinates for Bring's surface}

\author{Doha Kattan$^1$}
 \address{$^1$Department of Mathematics, Faculty of Science and Arts, King Abdulaziz University, Rabigh, Saudi Arabia}
\author{David Singerman$^2$}
 \address{$^2$School of Mathematical Sciences, University of Southampton, UK}
 \email{dakattan@kau.edu.sa, d.singerman@soton.ac.uk}

\maketitle

\section{Introduction} The Farey map $\mathbb{F}$, defined in section 2, is the universal triangular map, in that every triangular map on an orientable surface is a quotient of $\mathbb{F}$ by a subgroup of the modular  group, $\Gamma$ and every regular triangular map is a quotient of $\mathbb{F}$ by a normal subgroup of $\Gamma.$  This gives rise to many important Riemann surfaces such as the Klein surface.  In \cite{ISS} we used the Farey map to study the Klein surface.  In particular we showed how the Farey structure could be used to construct the famous 14-sided fundamental polygon that represents the Klein surface. (Other Riemann surfaces coming out of triangular maps are described in \cite{SS16}). In this paper we describe an important Riemann  surface that comes from a regular 4-gonal map,  namely Bring's surface.  We now use the universal 4-gonal map and the modular group is replaced by the Hecke group $H_4$.  Our main  result is to use Hecke-Farey coordinates to describe the 20-sided polygon that represents Bring's surface and describe the side pairings
needed to construct the surface.  This polygon was also given by \cite{RR92} and discussed by \cite{BN12} but here we provide detailed proofs that this polygon and side pairings does indeed give Bring's surface.

 \section{Bring's surface} We start with the following result which comes from Conder and Dobcsanyi's classifications of regular maps of small genus~\cite{CD01}.
 
  \begin{theorem} \label{th1}There is a unique regular map of type $\{5,4\}$ of genus 4 with 120 automorphisms.  The automorphism group of this map here is the symmetric group $S_5$.
    \end{theorem}

  Underlying this map is a unique Riemann surface of genus 4 with automorphism group $S_5$.  This surface is called {\it Bring's Surface}.  (Type $\{5, 4\}$ means that all faces are quadrilaterrals and all vertices have valency 5.)

We let $\Gamma(2, 5, 4)$ denote the triangle group with periods 2, 5, 4.  This triangle group acts as a group of conformal  automorphisms of the upper-half complex plane $\mathbb{H}$ and it has a presentation  
 
$$\Gamma(2,5,4)=\langle X,Y,Z\,|\,X^{2}=Y^{5}=Z^{4}=XYZ=1\rangle.$$
  \vspace{0.2cm}

  There is an epimorphism $\theta:\Gamma(2,5,4)\mapsto S_5$ given by $X\mapsto x=(1, 5),   Y\mapsto y=(5,4,3,2, 1)$,  and $Z\mapsto z=(2, 3, 4, 5).$
  
  The kernel $K$ of this epimorphism is a surface group of index 120 in $\Gamma$ and so by the Riemann-Hurwitz formula $\mathbb{H}/K$ is a Riemann surface of genus 4 with $S_5$ as automorphism group and so by Theorem \ref{th1}, this surface is Bring's Surface.
  
  As $K$ is a normal subgroup of the triangle group $\Gamma(2, 5, 4)$, it follows from~\cite{JS78} that underlying  Bring's surface  there is a regular map of type $\{5, 4\}$ which we consider in more detail later.

  \vspace{0.2cm}

  All compact Riemann surfaces arise from compact algebraic curves.  Bring's surface is the underlying Riemann surface of  {\it Bring's Curve} which is the complete intersection of of the three hyper-surfaces
\begin{equation}
\sum_{i=1}^{5} x_i = 0\;\;\;\; \sum_{i=1}^{5} x_i^{2} = 0\;\;\;\;\sum_{i=1}^{5} x_i^{3} = 0.
\end{equation} 
  \vspace{0.2cm}

Another way of considering Bring's Surface is  in terms of the Hecke group $H_4$.

Consider a group with the two generators $T: z\mapsto z+\lambda$ and $S: z\mapsto \frac{-1}{z}$.

Hecke showed that this group is discrete, i.e., a Fuchsian group, if and only if

(i) $\lambda=\lambda_q=2\cos \frac{\pi}{q}$ or 

(ii) $\lambda\ge 2$.
  \vspace{0.2cm}

We are interested in case (i).\\

We then find that $$R(z)=TS(z)=\frac{\lambda_q z-1}{z}.$$

This transformation is represented by the matrix 
$$\begin{pmatrix}\lambda_q&-1\\1&0\\\end{pmatrix}.$$

The eigenvalues of this matrix are $e^{\pm \pi i/q}$ and so $R$ has period $q$.  It is known that as a group  $H_q\cong C_2*C_q$.
  \vspace{0.2cm}

In this paper we are interested in the case  $q=4$ and then $\lambda_4=\sqrt 2$.
  The group $H_4$ has two generators 
  
 $$T:z\mapsto z+\sqrt 2 \;\;\;\;and\;\;\;\;  S: z\mapsto \frac{-1}{z}.$$
 
  Writing the M\"obius transformations as matrices we find that the transformations of $H_4$ are of two types\\
 (i) $$\begin{pmatrix}a&b\sqrt 2\\c\sqrt 2& d\\ \end{pmatrix}  \;  a,b,c,d \in \mathbb{Z} \;,\; ad-2bc=1$$ \\ and\\(ii) $$\begin{pmatrix}a\sqrt 2&b\\c&d\sqrt 2\\ \end{pmatrix}\;  a,b,c,d \in \mathbb{Z} \;,\; 2ad-bc=1.$$
 
 The matrices of type (i) form a subgroup of index 2 in $H_4$ and are called {\it even} matrices.  We denote this subgroup by $H_{4}^{e}$.
 The matrices of type (ii) are called {\it odd} matrices and they form a coset.

 \section{The congruence subgroups of the Hecke groups}
Let $I$ be an ideal of  $\mathbb Z[\lambda_{q}]$.\\We define
 \begin{equation*}
  PSL(2,\mathbb Z[\lambda_{q}],I) = \left\{ \left(
  \begin{matrix}
    a & b \\ c & d 
  \end{matrix}
  \right) \in PSL(2,\mathbb Z[\lambda_{q}]) \; |\;  a-1,b,c,d-1\;\in I
  \right\}.
\end{equation*}

Now for any ideal $I$ of  $\mathbb Z[\lambda_{q}]$ we define the $\mathit{ principal\; congruence \;subgroup}$ of the Hecke group $H_{q}$ as, $$H_{q}(I)=PSL(2,\mathbb Z[\lambda_{q}],I) \cap H_{q},$$ that is, the subgroup of $H_{q}$ consisting of elements in $PSL(2,\mathbb Z[\lambda_{q}])$.  For our interest we take the special case when $I=(n)$ where $2\leq n\in \mathbb{Z}^+$, i.e. $(n)$ is a principal ideal of $\mathbb{Z}[\sqrt{m}]$ where $m=2,3$ for $q=4,6$. The even principal congruence subgroup is now    \begin{equation}
  H_q^e(n) = \left\{ \left(
   \begin{matrix}
    a    & b \sqrt{m} \\
   c \sqrt{m}     & d \\
  \end{matrix}
  \right) \in H_q^e \; |\; a\equiv d \equiv \pm 1 \;  \text{mod}\; n, \; b\equiv c \equiv  0 \;   \text{mod}\; n
  \right\}.
  \label{h:1}
\end{equation}  
As  $a\sqrt{m}$ and $d\sqrt{m}$  are not congruent to $ \pm 1$ mod $n$ we can not have odd elements in $H_q(n).$

\vspace{1cm}

For $q=$4 or 6 and for $n>2$, the index of $H_q(n)$ in $H_q$ was found by Parson~\cite[Theorem 2.3]{Par76}.
\vspace{1cm}

\begin{theorem}\label{p} 
 \begin{equation}
\label{dk1c}
\mu_q(n)=|H_q:(H_q)(n)|= \begin{cases}
                           n^3   \prod_{p|n} (1-\frac{1}{p^2})             &if \; (n,m)=1\\
               n^3 (1-\frac{1}{m})   \prod_{p|n,p\neq m} (1-\frac{1}{p^2})               &if \; (n,m)=m.
           \end{cases}
\end{equation}
 
\end{theorem}   
\vspace{1cm}

Now we are interested in the quotient spaces $\mathbb{H}/ H_{4}(n)$. \\  
The fundamental region $F_{\lambda_4}$ for $ H_{4}$ is the region of the upper-half plane bounded by the unit circle and the lines $Re(z)=\pm{\sqrt 2/2}.$

\begin{figure}[H]
  \caption{The fundamental region $F_{\lambda_4}$ for $H_{4}$ }
    \includegraphics[width=0.6\textwidth]{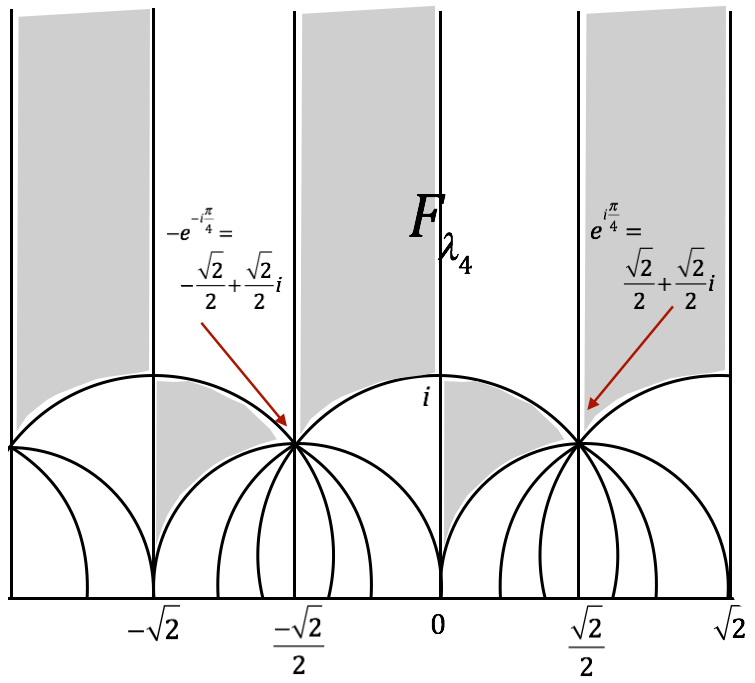}
    \centering
    \label{f9}
\end{figure}

This fundamental domain is not compact  and in fact the quotient space $\mathbb{H}/H_4$ has one puncture corresponding to the point at infinity in the fundamental domain. Thus  $\mathbb{H}/H_4(5)$ is a punctured surface.   However its compactification  $\overline {\mathbb{H}/H_4(5)}$ (found by adding the punctures) is a compact Riemann surface. Now by Theorem \ref{p},   $H_4(5)$ is a normal subgroup of index 120 in $H_4$ and so $\overline {\mathbb{H}/H_4(5)}$  is a compact Riemann surface with 120  automorphisms.
   \vspace{1cm}

   We now show its genus is equal to 4.  $H_4$ is the $\Gamma(2, \infty, 4)$ triangle group so that by \cite{JS78}  $\mathbb{H}/H_4(5)$ has the structure of a topological map. In fact, as $H_4(5)$ is a normal subgroup of $H_4$ (of index 120  by Theorem \ref{p}) it is a regular map with 120/2 =60 edges, 120/4=30 vertices and 120/5=24 faces. Thus if $g$ is the genus then $2-2g=30-60+24= -6,$ giving $g=4.$
      \vspace{1cm}

As $\mathbb{H}/H_4(5)$ has 120 automorphisms then by Theorem \ref{p} we deduce
\vspace{1cm}

 \begin{theorem}  
 $\overline {\mathbb{H}/H_4(5)}$ is Bring's surface.
   \end{theorem}

 



 \section{Universal $q$-gonal maps}

  In \cite{Sin88, KS22} it was shown that the Farey map $\hat{\mathscr{M}}_{3}$ is the universal triangular map.  The Farey map is a map on the upper-half plane whose vertices are the extended rationals $\mathbb{Q}\cup \{\infty\}$
   and where two rationals $\frac{a}{c}$ and $\frac{b}{d}$ are joined by an edge if and only if $ad-bc=\pm 1.$ We sometimes refer to the  rational $\frac{a}{c}$ as a {\it Farey coordinate.} The automorphism group of $\hat{\mathscr{M}}_{3}$ is the classical modular group $\Gamma.$ Another way of constructing $\mathbb{F}$ is as the map whose vertex set is the extended rationals (the images of $\infty$ under $\Gamma$)  and whose edge set is the images of the imaginary axis under $\Gamma.$

      In  \cite{KS22} the authors constructed the universal $q$-gonal maps $\hat{\mathscr{M}}_{q}.$  The vertices of this map are the images of $\infty$ under $H_q$ and the edges are the images of the imaginary axis under $H_q.$ In \cite{KS22} these universal maps were drawn for $q=4,5,6,7$.  The {\it principal face} $F_q$ has as its vertices, the images of $\infty$ under the $q$ powers of $R.$ \\ In this paper we are interested in Bring's map and so we are particularly interested in the case $q=4.$ The universal 4-gonal map is drawn in Figure \ref{3.9}.

The vertices of ${\hat{\mathscr{M}}_{4}}$ are the images of $\infty$ under $H_4$  and these are all real numbers of the form $\frac{a\sqrt 2}{c}$ or   
 $\frac{b}{d\sqrt 2}$ where $a, b, c ,d\in \mathbb{Z}$  and $(a,c)=(b,d)=1.$  We call $\frac{a}{c\sqrt 2}$ and $\frac{b\sqrt 2}{d}$ as {\it Hecke-Farey coordinates.} An edge joins $\frac{a\sqrt 2}{c}$ and   $\frac{b\sqrt 2}{d}$ if and only if $ad-2bc=\pm 1.$ The {\it principal face} $F_4$ is the $4$-gon whose vertices are $\infty$, $\frac{0}{1}, \frac{1}{\sqrt 2}, \frac{\sqrt 2}{1}.$
 
\begin{figure}[H]  
  \caption{Principal face when $q=4$}
    \includegraphics[width=0.6\textwidth]{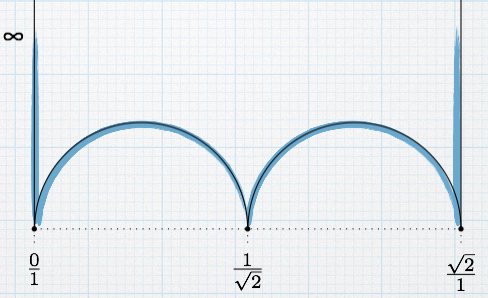}
    \centering
    \label{3.2.b}
\end{figure}

\begin{landscape}

\begin{figure}[!h]
    \includegraphics[scale=0.7]{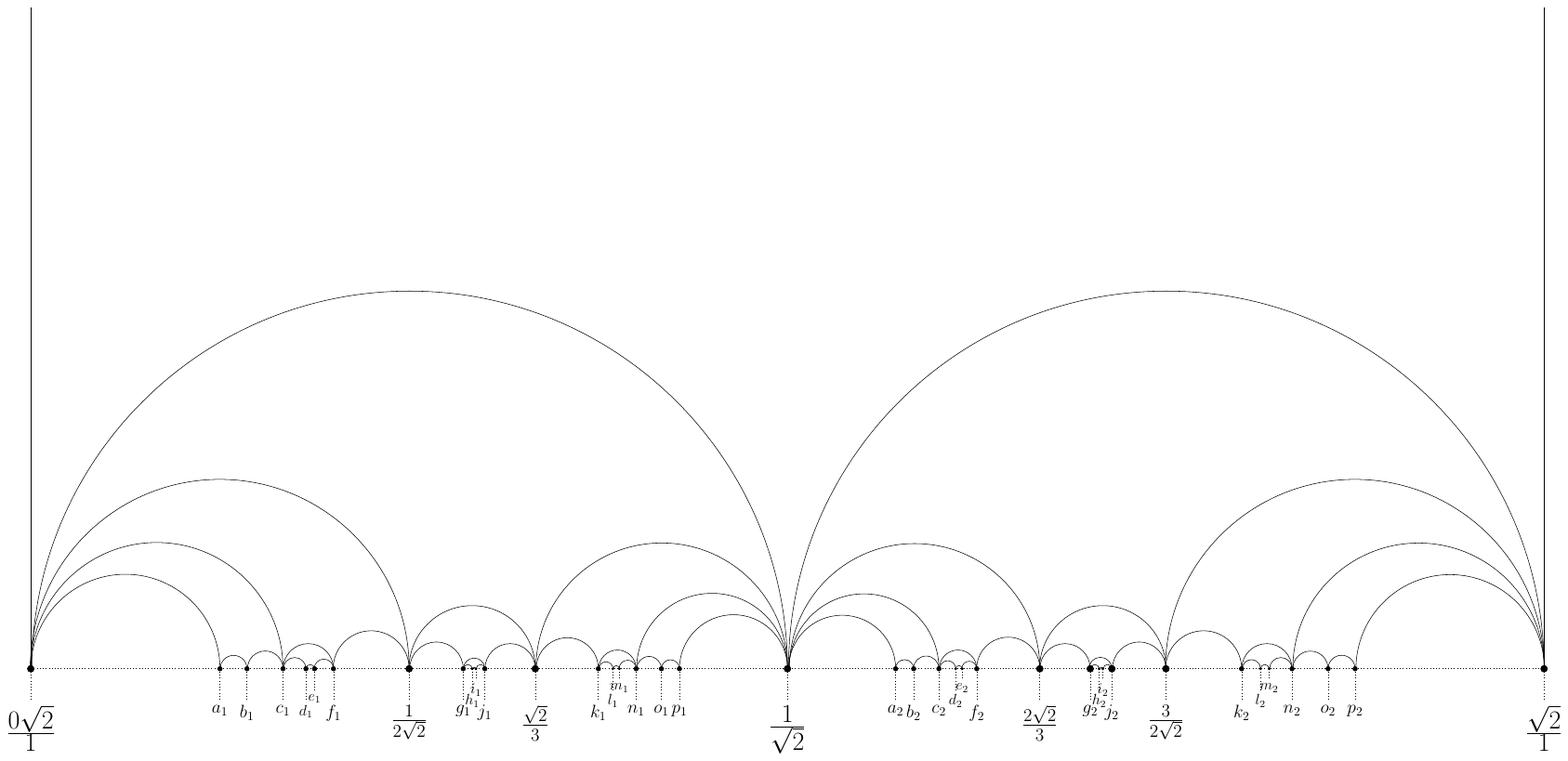}
        \vspace{-15mm}

    \centering
      \caption{The universal 4-gonal tessellation $\hat{\mathscr{M}}_{4}$ }
\label{3.9}
\end{figure}
\end{landscape}
\newpage

 \section{Farey fractions modulo $n$} Farey fractions were first introduced in \cite{SS19} for the case $q=3$.  A Farey fraction modulo $n$ is a pair $(a,c)$ where $(a,c,n)=1$ and where we identify $\frac{a}{c}$ with $\frac{-a}{-c}.$  We think of these fractions as being points on $\mathbb{H}/\Gamma(n)$. Two Farey fractions mod $n$,  $(a,c)$  and $ (b,d)$  are adjacent when $ad-bc\equiv \pm 1  \; \text{mod}\;n$. 
 
 \begin{example} {$n=5.$} The Farey fractions mod 5 are 
   
   $$\frac{1}{0}, \frac{2}{0}, \frac{0}{1}, \frac{1}{1}, \frac{2}{1}, \frac{3}{1}, \frac{4}{1}, \frac{0}{2},\frac{1}{2} , \frac{2}{2}, \frac{3}{2} , \frac{4}{2}.$$
   
   We join $\frac{a}{c}$ and $\frac{b}{d}$ by a edge if and only if $ad-bc\equiv \pm 1 mod\ 5$

   \end{example}

There are 12 vertices and by drawing the diagram we see that we get a regular icosahedron.   
\begin{figure}[H]
  \caption{Drawing of $\mathscr{M}_{3}(5)$ with Farey coordinates \cite{ISS} }
    \includegraphics[width=0.5\textwidth]{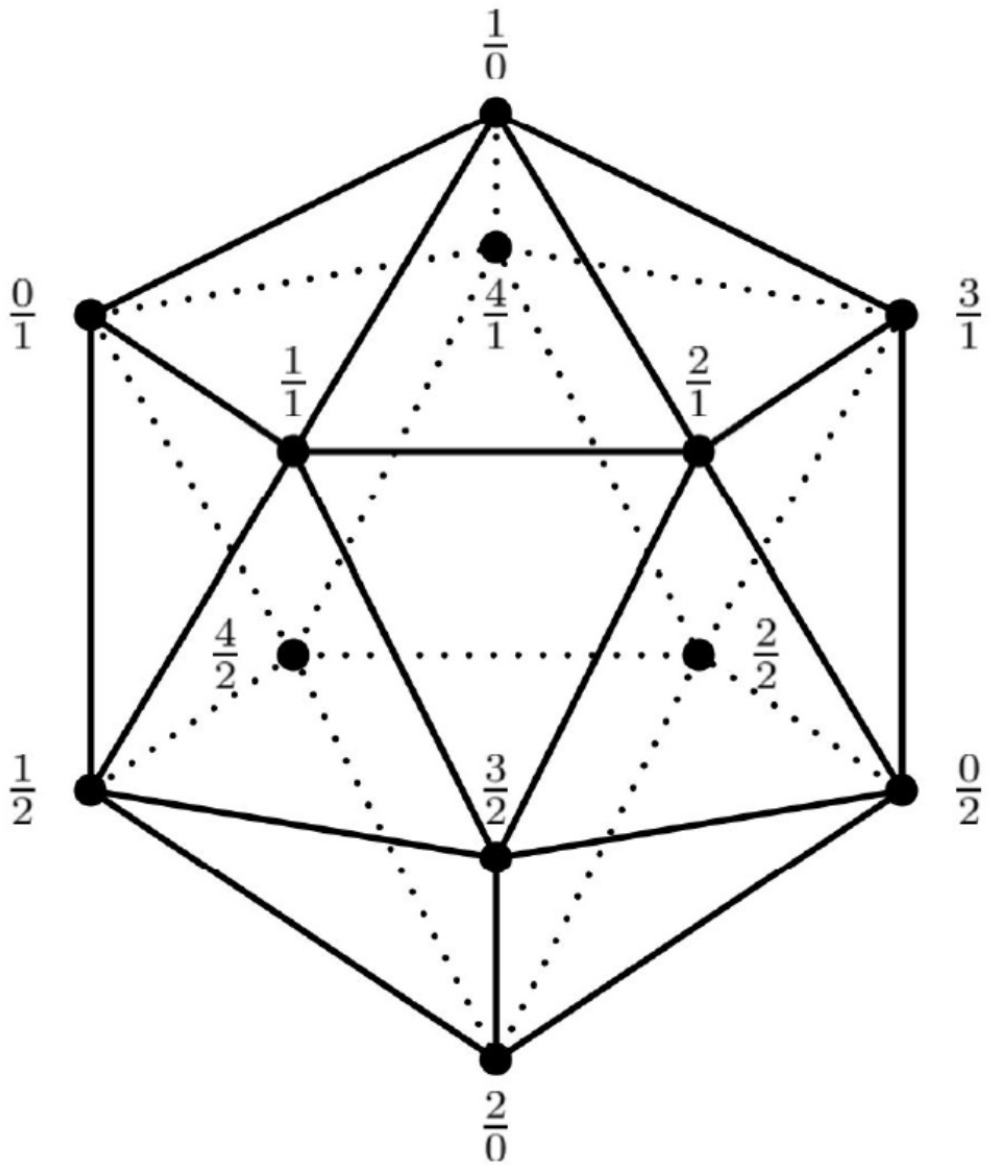}
    \centering
    \label{f9}
\end{figure} 

\section{Hecke-Farey fractions for the cube} 
  
 The main aim of this paper is to find Hecke-Farey fractions for Bring's curve, $\mathscr{M}_4(5).$   In general, for a positive integer $n$  the map $\mathcal{M}_4(n)$ has vertices of the form $\frac{a}{\sqrt 2 c}$ or $\frac{b\sqrt 2}{d}$. where $a,b,c,d\in \mathbb{Z}_n.$   These numbers are called {\it Hecke-Farey coordinates.}These two vertices are joined by an edge If and only if $ad-2bc\equiv \pm 1$ mod$\ n.$
 
 Before we consider $\mathscr{M}_4(5)$  we look at the simpler case $\mathscr{M}_4(3)$ which gives the cube.

 We are now considering the triangle group $\Gamma(2,3,4)$. This has a presentation 
 $\langle{X,Y,Z| X^2= Y^3= Z^4=XYZ=1}\rangle.$ It is well-known that $\Gamma(2,3 ,4)$ is isomorphic  to the symmetric group $S_4,$ the rotation group of the cube.  In fact, $X\mapsto (1,2), Y\mapsto(1, 3,4),  Z\mapsto (1,2,4,3) $ gives this isomorphism. 
 Now from (\ref{dk1c}), $|H_4:H_4(3)|=24$ so that $\mathscr{M}_4(3)$ has 24 darts, 12 edges,  8 vertices, and 6 faces. $\mathscr{M}_4 (3)$ is a regular map, it must be the cube. We now see that there are exactly 8 Hecke-Farey fractions modulo 3.
 
 These are $$ \frac{1}{0\sqrt 2}, \frac{0}{1\sqrt 2},  \frac{1}{1\sqrt 2},  \frac{1}{2\sqrt 2},      \frac{0\sqrt 2}{1},  \frac{2\sqrt 2}{1} , \frac{1\sqrt 2}{0},   \frac{1\sqrt 2}{1}.$$
 
 (Note for example that $\frac{\sqrt 2}{2}=\frac{-\sqrt 2}{-2}=\frac{2\sqrt 2}{1}$ so that this is the complete list of Hecke--Farey fractions modulo 3.)
 
 By drawing a diagram we notice that the corresponding Farey map is a cube.

 \begin{figure}[H]

     \centering

   \includegraphics[width=0.4\textwidth]{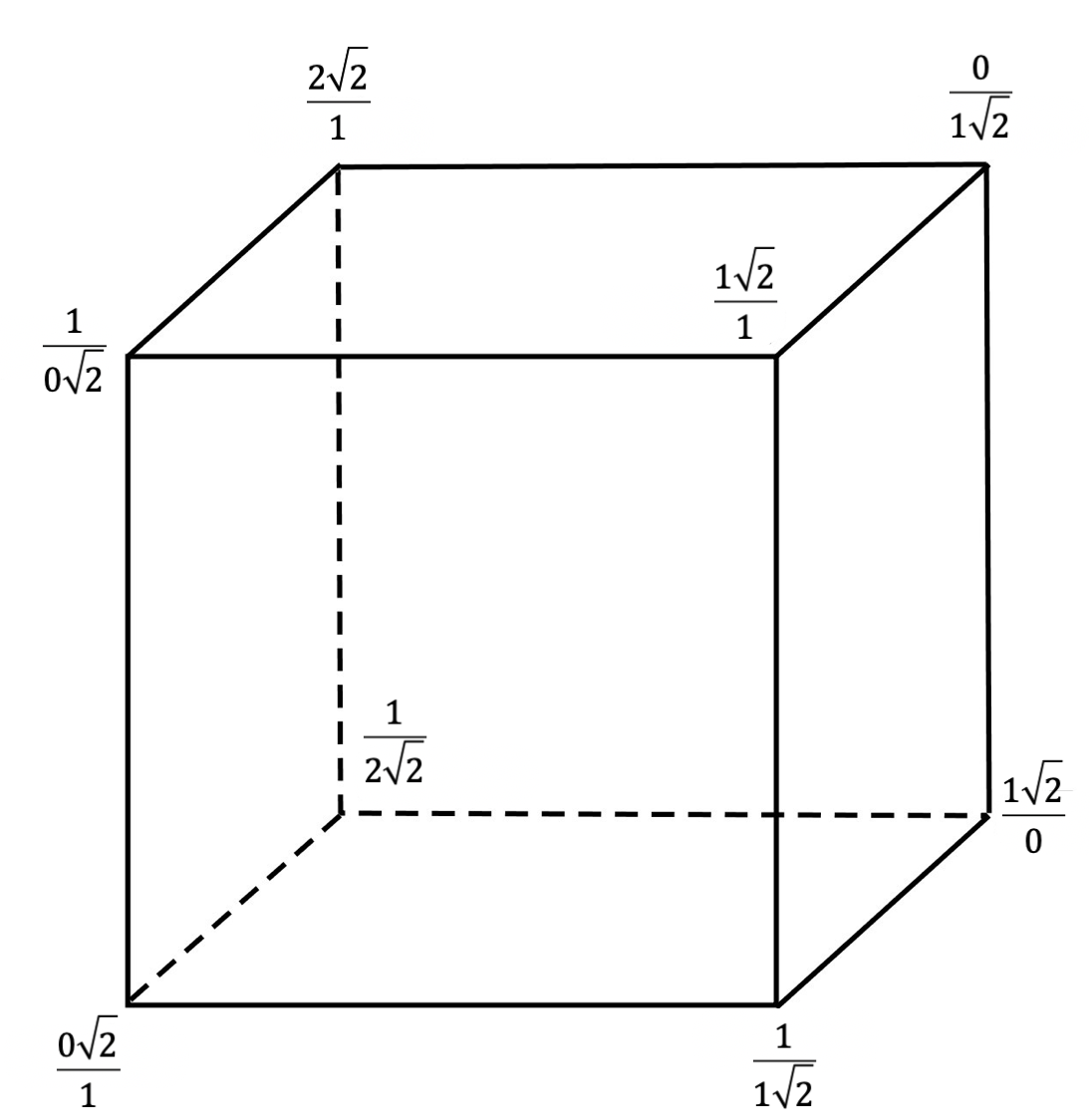}
    \caption{$\mathscr{M}_{4}(3)$; Cube embedded in a sphere}    
    \label{cube}
      \end{figure}

\section{The Bring curve $\mathscr{M}_4(5)$}
We want want to give each   vertex of  $\mathscr{M}_4(5)$ a Hecke-Farey coordinate.  These will have the form $\frac{a}{c\sqrt 2}$ or $\frac{b\sqrt 2}{d}$ where $a,b,c,d$ lie in the set $\{0,1,2,3,4\}$. Now $\frac{a}{c\sqrt 2} = \frac{x}{y\sqrt 2}$ if  and only if $a=\pm x$ and $c=\pm y$, etc. so we can easily work out all the 24 Hecke- Farey fractions representing the points of $\mathscr{M}_{4}(5)$. These are presented in Table \ref {tab:tab6}.

\begin{table}[H]
\caption {Table of Correspondence for $\mathscr{M}_{4}(5)$ } \label{tab:tab6} 

\centering
 \begin{tabular}{lc lc lc l} 
 \midrule 

$ A_1:\frac{1}{0\sqrt{2}}$ & $ B_1:\frac{2}{0\sqrt{2}}$ & $ C_1:\frac{0}{1\sqrt{2}}$ & $ D_1:\frac{1}{1\sqrt{2}}$ \\  
 \midrule 

 $ E_1:\frac{2}{1\sqrt{2}}$ & $ F_1:\frac{3}{1\sqrt{2}}$ & $ G_1:\frac{4}{1\sqrt{2}}$ & $ H_1:\frac{0}{2\sqrt{2}}$ \\ 
 \midrule 

$ I_1:\frac{1}{2\sqrt{2}}$ & $ J_1:\frac{3}{2\sqrt{2}}$ & $ K_1:\frac{2}{2\sqrt{2}}$ & $ L_1:\frac{4}{2\sqrt{2}}$  \\
 \midrule 

 $ A_2:\frac{0\sqrt{2}}{1}$ & $ B_2:\frac{0\sqrt{2}}{2}$ & $ C_2:\frac{1\sqrt{2}}{0}$ & $ D_2:\frac{1\sqrt{2}}{1}$ \\
 \midrule 

  $ E_2:\frac{1\sqrt{2}}{2}$ & $F_2:\frac{1\sqrt{2}}{3}$ & $ G_2:\frac{1\sqrt{2}}{4}$ & $ H_2:\frac{2\sqrt{2}}{0}$ \\
 \midrule 
 $ I_2:\frac{2\sqrt{2}}{1}$&$ J_2:\frac{2\sqrt{2}}{2}$&$ K_2:\frac{2\sqrt{2}}{3}$&$ L_2:\frac{2\sqrt{2}}{4}$ \\

 \hline
 \end{tabular}
\end{table}

\section{Poles} In a Farey map (the case $q=3$) a {\it pole} is a vertex with a Farey coordinate of the form $\frac{a}{0}$, where $a\not =0.$ For example, in the icosahedron $\mathscr{M}_{3}(5)$ the poles are $\frac{1}{0}$ and $\frac{2}{0}.$ which from Figure \ref{f9} are the north and south poles of the icosahedron.  Also, in \cite{ISS} Farey fractions were used to construct the fundamental region for the Klein quartic.  We found that $\frac{1}{0}$ was at the centre whilst $\frac{2}{0}$ and $\frac{3}{0}$ lie on the boundary. This is a typical situation. The pole $\frac{1}{0}$ always lies at the centre and the other poles lie on the boundary. In a rough sense these other poles lie at "infinity". The poles of $\mathscr{M}_{4}(5)$ are
$$A_1= \frac{1}{0\sqrt 2},  \, B_1=\frac{2}{0\sqrt 2},  \,  C_2=\frac{\sqrt 2}{0}, \, H_2=\frac{2\sqrt 2}{0}.$$

   \begin{figure}[H] 
    \centering

    \includegraphics[width=0.9\textwidth]{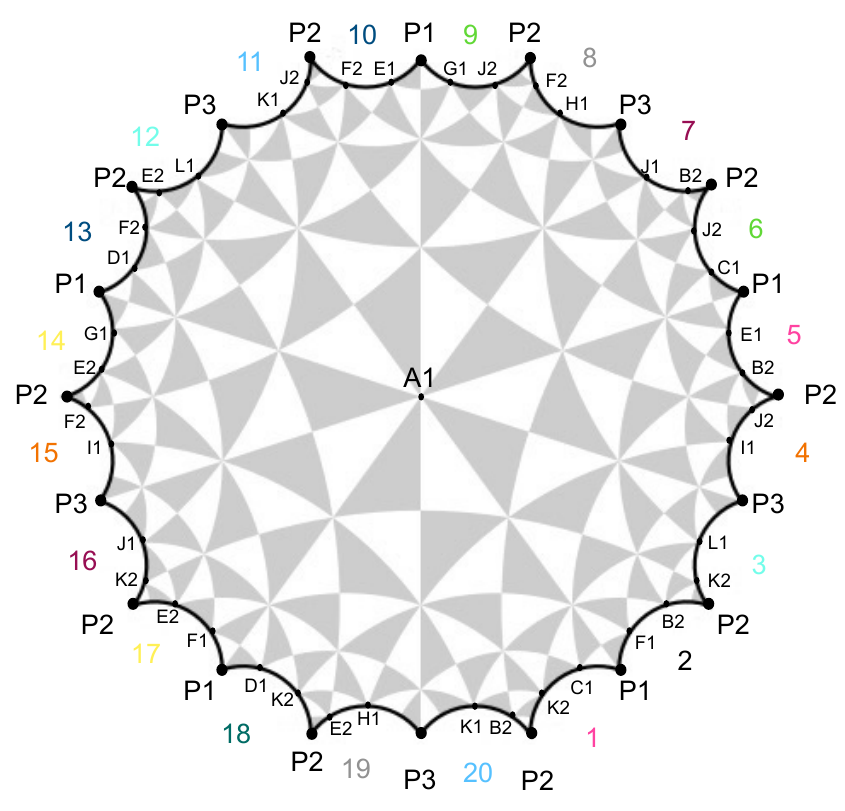}
    \caption{}     
    \label{art123}
\end{figure}



   \section{Hecke--Farey coordinates for Bring's surface}
There are two important papers where Bring's surface is drawn as a 20-sided polygon with side identifications. The first of these is by G.~Riera and R.~E.~Rodriguez \cite{RR92} from 1992.  The second is by H.~W.~Braden and T.~P.~Northover \cite{BN12}, who just used the Riera--Rodriguez diagram. 
   This diagram is slightly clearer so we use this. This is given in Figure \ref{art123}, but without the Hecke--Farey coordinates $G_1, J_2$ etc. In the diagram we also use their notation for the poles, So $P1 = H2,\;P2 = B2, \;P3 = C2$. Otherwise we use the original notation.
   
Riera and Rodriguez  did not give a proof that that their polygon represents Bring's surface. They write that the consider a Cayley diagram for $S_5$ and embed it into the surface but no details are given, nor did  they  justify the side pairings. Our aim is to give Hecke--Farey coordinates for the vertices of the 20-sided polygon.  This will tell us how to draw the polygon and also how to find the side-pairings. We start with a simple observation.

\begin{lemma} The map $z\mapsto z+\sqrt 2$ of the Hecke-Farey coordinates induces a rotation of $\mathscr{M}_4(5)$ through $2\pi/5$ about the centre $\frac{1}{0}.$
 
\end{lemma}

\begin{proof} As we are working modulo 5 in in $\mathscr{M}_4(5)$ the transformation has order 5 and fixes the centre $\frac{1}{0}.$ (Here we are identifying  the Hecke-Farey coordinate with its corresponding vertex.) 
\end{proof}

For example, consider the pole $H_2=\frac{2\sqrt 2}{0}$, which we regard as the top vertex of our 20-sided polygon (See Figure \ref{art123}).  As $H_2+\sqrt 2=H_2$ we find that there are 5 vertices of our polygon, labelled $H_2.$  
The same applies to the pole $C_2$. There are 5 poles with label $C_2.$ However,  as we now see, there are 10 poles with label $B_1$.

\begin{definition} A  {\it Farey circuit}  is a collection of points $v_0, v_1,\cdots v_n$ each with a Hecke-Farey coordinate such that $v_i$ is adjacent to $v_{i+1}$ and $v_0=v_n.$

\end{definition}

A Farey circuit on $\mathscr{M}_4(5)$ is given by $H_2,  E_1, F_2, B_1, J_2, K_1, C_2, L_1, E_2, B_1,\\ F_2, D_1, H_2$ (12 vertices). For example $H_2=\frac{2\sqrt 2}{0}$ is adjacent to $E_1=\frac{2}{\sqrt 2}$ and $E_1$ is adjacent to $F_2=\frac{\sqrt 2}{3}$, etc.  We continue this Farey sequence  just by adding multiples of  $\sqrt 2.$ For example, if we add  $\sqrt 2$ to our circuit  we get $H_2+\sqrt 2=H_2$,  $E_1+\sqrt 2=G_1$, $F_2+\sqrt 2=E_2$ etc., and we  get another 12 vertices. We get 12 more vertices each by adding $2\sqrt 2, \;3\sqrt 2$ and $4\sqrt 2$. In  this way we find Hecke-Farey  coordinates for all of the  60 boundary vertices as illustrated in Figure \ref{art123}.  Now in the first 12 vertices we found that there was one pole labelled $H_2=P_1$, one pole labelled $C_2=P_3$  and two poles labelled $B_2=P_2$.  Thus in total there are 5 poles of the form $P_1$, 5 poles labelled $ P_3$ and 10 poles labelled $P_2$. Thus in total there are 20 poles on the boundary and we regard these as the vertices of a 20-sided hyperbolic polygon. We use the diagram in \cite{BN12}, where the edges are labelled $1,2,\cdots , 20.$

In the diagram of Bring's surface with its  Hecke-Farey coordinates, each at the 20 edges has one vertex of valency 5 in its interior.
 For example, edge 1 has $K2$ as the vertex of valency 5 in its interior. In Table \ref{table:ds} we just list the edges together with these vertices.
 
 \begin{table}[h!]   
\begin{center}
\begin{tabular}{ | c | c | } 
\hline
 $1 \leftrightarrow  K_2$ &  $2 \leftrightarrow B_2$   \\ \hline
 $3 \leftrightarrow  L_1$ &  $4 \leftrightarrow I_1$  \\ \hline
 $5 \leftrightarrow  B_2$ &  $6 \leftrightarrow J_2$ \\ \hline
$7 \leftrightarrow  J_1$ &  $8 \leftrightarrow H_1$\\ \hline
$9 \leftrightarrow  J_2$ &  $10 \leftrightarrow F_2$ \\ \hline
$11 \leftrightarrow  K_1$ &  $12 \leftrightarrow L_1$ \\ \hline
$13 \leftrightarrow  F_2$ &  $14 \leftrightarrow E_2$ \\ \hline
$15 \leftrightarrow  I_1$ &  $16 \leftrightarrow J_1$ \\ \hline
$17 \leftrightarrow  E_2$ &  $18 \leftrightarrow K_2$ \\ \hline
$19 \leftrightarrow  H_1$ &  $20 \leftrightarrow K_1$ \\ \hline

\end{tabular}
    \end{center}
    \caption{}
    \label{table:ds}

\end{table}   

 We now just pair the sides of the polygon which have the same Hecke-Farey coordinate.  For example both side 1 and side 18 have the vertex $K_2$ so we pair them.
 
 We give all the side pairings in Table \ref{table:c}.
 \begin{table}[h!]   
\begin{center}
\begin{tabular}{ | c | c | } 
\hline
 $2 \leftrightarrow  5$ &  $6 \leftrightarrow 9$   \\ \hline
 $10 \leftrightarrow  13$ &  $14 \leftrightarrow 17$  \\ \hline
 $18 \leftrightarrow  1$ &  $3 \leftrightarrow 12$ \\ \hline
$7 \leftrightarrow  16$ &  $11 \leftrightarrow 20$\\ \hline
$15 \leftrightarrow  4$ &  $19 \leftrightarrow 8$ \\ \hline

\end{tabular}
    \end{center}
    \caption{}
    \label{table:c}

\end{table}   

We notice that if $k\equiv 2\,
\; \text{mod}\; 4$ then we pair $k$ with $k+3$, while if $k\equiv  3\; \text{mod}\; 4$ then we pair $k$ with $k+9$.  This is similar, but not exactly the same as in \cite{RR92} or \cite{BN12} where they give their answers without  proof.

\newpage

\section{Equivalence  classes  of vertices}.
  
  With these side identifications we can easily find the equivalence classes of vertices. \\Label the vertices $a_1, a_2,\cdots a_{20}$ with edge 1 having vertices  $a_1, a_2$ (going anti-clockwise), side 2 having vertices $a_2, a_3$ etc.

As 2 is paired to 5, we get $a_2=a_6$,\;$a_3=a_5$.

As 6 ia paired to 9, we get $a_6=a_{10}$,  $a_7=a_9.$

As 7 is paired with 16, we get $a_7=a_{17},\; a_8=a_{16}$.

As 10 is paired to 13, we have $a_{11}=a_{13}$, $a_{10}=a_{14}.$

As 11 is paired to 20 is paired with 13, we get $a_{11}= a_{13}$, $a_{10}=a_{14}.$

As 14 is pared with 17 we gat $a_{15}=a_{17}$, $a_{14}=a_{18}.$

As 4 is paired to 15 $a_4=a_{16}$, $a_3=a_{15}.$

As 18 is paired to 1, we get $a_1=a_{19}$, $a_{18}=a_2.$

As 8 is paired with 19 , we get $a_8=a_{19}$, $a_8=a_{20}.$

Thus we find that $a_k$ where $k$ is odd represent one vertex in $\mathscr{M}_4(5)$,   $a_2,\; a_6,\;a_{10},\;a_{14},\;a_{18}$ represent another vertex and $a_4,\;a_8,\;a_{12}$ and $a_{16}$ represent a third vertex. Thus $\mathscr{M}_3(5)$, Bring's surface has 3 vertices, 10 edges and 1 face.  We then get $2-2g=3-10+1= -6$ dgiving $g=4$ which it should be.


.

\end{document}